\documentclass[fleqn,preprint,3p,a4paper]{elsarticle}

\usepackage{amsmath,amssymb,amsthm,mathtools}
\usepackage[T1]{fontenc}
\usepackage{lmodern,microtype}
\usepackage{booktabs,mathrsfs,enumerate}
\usepackage[hidelinks]{hyperref}

\newtheorem{theorem}{Theorem}[section]
\newtheorem{lemma}[theorem]{Lemma}
\newtheorem{proposition}[theorem]{Proposition}
\newtheorem{corollary}[theorem]{Corollary}
\theoremstyle{definition}
\newtheorem{remark}[theorem]{Remark}

\newcommand{\N}{\mathbb N}
\newcommand{\HH}{\mathbb H}
\newcommand{\RR}{\mathcal R}
\newcommand{\rows}{\operatorname{rows}}

\begin{document}
\begin{frontmatter}

\title{Scott products and sobriety of countable meet-continuous dcpos\tnoteref{t1}}
\tnotetext[t1]{This research was supported by the National Natural Science Foundation of China  (Nos. 12471070, 12071199, 12561083).}

\author[1]{Xiaoquan Xu\corref{cor1}}
\cortext[cor1]{Corresponding author}
\ead{xiqxu2002@163.com}
\author[2]{Wei Ji}
\ead{javeey@163.com}

\address[1]{School of Computer Information Engineering, Nanchang Institute of Technology, Nanchang 330044, China}
\address[2]{School of Mathematics and Statistics, Guilin University of Technology, Guilin 541004, China}

\begin{abstract}
We study Scott-product compatibility and sobriety for countable meet-continuous directed-complete partial orders (dcpos). We construct a countable meet-continuous dcpo $\RR$ whose Scott space is well-filtered but not sober, although the Scott topology on every finite power is the product of the factor Scott topologies. The ideals of $\RR$ admit a unique finite-parameter description. We also construct a countable meet-continuous dcpo $\HH$ with a greatest element whose Scott space is coherent and well-filtered, but whose Scott square has a topology strictly finer than the ordinary product topology. These constructions answer negatively two questions about countable meet-continuous dcpos. On the positive side, we prove that a meet-continuous $L$-dcpo, meaning a dcpo whose principal ideals are complete lattices, is Scott sober whenever its Scott square carries the ordinary product topology. This answers Jia's core-compactness question affirmatively for $L$-dcpos without any countability assumption. Finally, we characterize well-filteredness of $L$-dcpos by boundedness of closed Rudin sets and formulate related open questions.
\end{abstract}


\end{frontmatter}

\section{Introduction}\label{sec:intro}
The Scott topology connects directed order convergence with topological continuity. Two basic questions ask whether the Scott topology on an order product agrees with the product of the factor Scott topologies, and whether every irreducible Scott closed set has a generic point. Xu and Ji proved both conclusions for countable meet-continuous complete lattices~\cite[Theorem~1.1]{XuJi}. Xu then asked whether they remain valid for arbitrary countable meet-continuous dcpos, without any least-element or semilattice assumption~\cite[Questions~1.1 and~1.2]{Xu}.

We answer both questions negatively. Section~\ref{sec:fan} constructs a countable meet-continuous row-tail dcpo $\RR$ whose Scott space is well-filtered but neither coherent nor sober, although every finite Scott power carries the ordinary product topology. Its ideals have a unique finite-parameter normal form. Section~\ref{sec:word} constructs a countable meet-continuous finite-word dcpo $\HH$ with a greatest element such that
\[
 \sigma(\HH)\times\sigma(\HH)\subsetneq\sigma(\HH\times\HH).
\]
Its Scott space is coherent and well-filtered. Thus finite-power product equalities do not imply sobriety, and coherence together with well-filteredness does not ensure the square product equality, even in the countable meet-continuous class.

The constructions exhibit different obstructions. For $\HH$, a common lower bound relation is Scott open in the order square but contains no ordinary open rectangle at the greatest pair. For $\RR$, a common upper bound section fails to be closed under directed suprema. These relations also guide the positive results.

Section~\ref{sec:local} proves that a meet-continuous $L$-dcpo---a dcpo whose principal ideals are complete lattices---is Scott sober whenever its square satisfies the Scott product equality (Theorem~\ref{local:sober}). Combined with Xu's local product theorem~\cite[Theorem~3.4]{Xu}, this yields sobriety for countable meet-continuous dcpos with local binary joins. The cardinality-free theorem also answers Jia's core-compactness question~\cite[Question~4.5.8, p.~139]{Jia} affirmatively for $L$-dcpos: core-compactness supplies the square equality by~\cite[Theorem~II-4.13]{GHK}. Corollaries~\ref{local:corecompact} and~\ref{local:functionspace} treat the question and its function-space version.

The row-tail pattern is motivated by Gruenhage and Streicher's pseudobase construction~\cite[pp.~225--226]{GS}. Its ideal normal form permits application of the product criterion of Miao, Xi, Li and Zhao~\cite[Lemma~4.1, pp.~827--829]{MXLZ}. A separate compactness argument proves well-filteredness: every compact subset lies in an algebraic finite-row truncation carrying the subspace Scott topology. For $\HH$, finite generation of common upper bound sets gives the strong $d$-space property and hence well-filteredness~\cite[Corollary~64]{LXLower}.

Section~\ref{sec:boundaries} compares the examples and formulates open questions. Proposition~\ref{local:rudin-boundedness} reduces well-filteredness of $L$-dcpos to boundedness of closed Rudin sets. The sobriety of $\Sigma\HH$ and the implication from sobriety to the square product equality remain open in the countable meet-continuous class.

\section{Preliminaries and common-bound relations}\label{sec:relations}
All posets and directed sets are nonempty, and countable sets may be finite. For $B\subseteq P$, write
\[
\begin{aligned}
 {\downarrow} B&=\{x\in P:x\leq b\text{ for some }b\in B\},\\
 {\uparrow} B&=\{x\in P:b\leq x\text{ for some }b\in B\}.
\end{aligned}
\]
We abbreviate singleton arguments. An \emph{ideal} is a directed lower set; ideals are not required to be Scott closed. A \emph{dcpo} is a poset in which every directed set has a supremum. A subset $U$ of a poset $P$ is said to be \emph{Scott open} if (i) $U=\mathord{\uparrow}U$, and (ii) for any directed subset $D$ whose supremum exists in $P$, $\sup D\in U$ implies $D\cap U\neq\emptyset$.  All Scott open subsets of $P$ form a topology, called the \emph{Scott topology} on $P$ and denoted by $\sigma(P)$. The space $\Sigma P=(P,\sigma(P))$ is called the \emph{Scott space} of $P$. Scott closed sets are precisely the lower sets closed under existing directed suprema. A nonempty Scott closed set $A$ is \emph{irreducible} if any two open sets meeting $A$ have an intersection meeting $A$. The space $\Sigma P$ is \emph{sober} if every such $A$ is ${\downarrow} a$ for a unique $a\in P$.

Following Kou, Liu and Luo~\cite{KLL}, a dcpo $P$ is called \emph{meet-continuous} if, for every $x\in P$ and every directed subset $D\subseteq P$,
\begin{equation}\label{def:mc}
 x\leq\sup D
 ~\Longrightarrow~
 x\in\operatorname{cl}_{\sigma(P)}
       ({\downarrow} x\cap{\downarrow}D).
\end{equation}
No binary meet, binary join, least element or way-below continuity is implicit. We call
\[
 \sigma(P\times Q)=\sigma(P)\times\sigma(Q)
\]
the \emph{Scott product equality} for $P$ and $Q$; the topology on the right is the ordinary product topology. The inclusion $\sigma(P)\times\sigma(Q)\subseteq\sigma(P\times Q)$ always holds because the coordinate projections preserve directed suprema.

Let $P$ be a dcpo. For $U\in\sigma(P)$, define the common lower bound relation
\[
 \mathrm{LB}_U
 =\{(x,y)\in P^2:\text{some }z\in U\text{ satisfies }z\leq x,y\}.
\]
Condition~\eqref{def:mc} is equivalent to Scott openness of $\mathrm{LB}_U$ for every Scott open $U$. For the forward implication, apply meet-continuity successively to the two coordinate projections of a directed subset of $P^2$ and then take a common upper member. Conversely, apply Scott openness to the directed set $\{(x,d):d\in D\}$ when $x\leq\sup D$, taking arbitrary Scott open neighbourhoods $U$ of $x$. We use this elementary reformulation in the finite-word construction.

For a nonempty irreducible Scott closed set $A$, put
\[
 \mathrm{UB}_A
   =\{(x,y)\in P^2:\text{some }a\in A\text{ satisfies }x,y\leq a\}.
\]
For a dcpo $P$ satisfying the square Scott product equality, Xu proved that $\Sigma P$ is sober if and only if ${\downarrow}({\uparrow}x\cap A)$ is Scott closed for every $x\in P$ and every nonempty irreducible Scott closed subset $A$ of $P$~\cite[Proposition~5.2]{Xu}. In the sufficiency proof, these sections make $\mathrm{UB}_A$ Scott closed in the order square, hence closed in the ordinary product. Irreducibility makes the diagonal of $A$ dense in $A\times A$. Section~\ref{sec:local} applies this common-upper-bound method using local joins.

In a lattice, $\mathrm{LB}_U$ is the inverse image of $U$ under binary meet, and $\mathrm{UB}_A$ is the inverse image of $A$ under binary join. Outside lattices the two relations need not have parallel behaviour.

For a $T_0$ space $X$, the specialization order is defined by
\[
 x\leq_X y\quad\Longleftrightarrow\quad x\in\operatorname{cl}_X\{y\}.
\]
Write $\mathcal Q(X)$ for the collection of its nonempty compact saturated subsets. Saturated subsets of a Scott space are precisely upper sets. A family $\mathcal K\subseteq\mathcal Q(X)$ is \emph{filtered} if it is nonempty and, for $K_1,K_2\in\mathcal K$, some $K_3\in\mathcal K$ satisfies $K_3\subseteq K_1\cap K_2$. The space $X$ is \emph{well-filtered} if, for every filtered family $\mathcal K\subseteq\mathcal Q(X)$,
\[
 \bigcap\mathcal K\subseteq U,~ U\text{ open}
 ~\Longrightarrow~
 K\subseteq U\text{ for some }K\in\mathcal K.
\]
It is \emph{coherent} if the intersection of two compact saturated subsets is compact. Every sober space is well-filtered~\cite[Theorem~II-1.21]{GHK}.

A nonempty closed subset $A$ of $X$ is a \emph{closed Rudin set} if, for some filtered $\mathcal K\subseteq\mathcal Q(X)$, it is an inclusion-minimal closed set meeting every member of $\mathcal K$~\cite{XZ}. These sets are also called closed $KF$-sets in~\cite{MLZ}. We recall the following standard characterization, including its short proof for the localization argument in Section~\ref{sec:local}.

\begin{lemma}[Rudin characterization, cf.~\cite{XZ}]\label{wf:rudin-characterization}
A $T_0$ space is well-filtered if and only if every closed Rudin set is the closure of a point. For a Scott space, this means that every closed Rudin set is a principal ideal.
\end{lemma}
\begin{proof}
First note that if a closed set $C$ meets every member of a filtered family $\mathcal K\subseteq\mathcal Q(X)$, then $C$ contains a closed Rudin set for that family. Indeed, the intersection of a nonempty inclusion chain of closed subsets of $C$ meeting every $K\in\mathcal K$ still meets every $K$, by compactness of $K$. Zorn's lemma supplies a minimal such closed set. This is the filtered-family case of the topological Rudin lemma~\cite{XZ}.

Suppose $X$ is well-filtered and $A$ is a closed Rudin set witnessed by $\mathcal K$. Then $A\cap\bigcap\mathcal K\ne\emptyset$, since otherwise well-filteredness applied to $X\setminus A$ would contradict the witnessing property. For $a$ in this intersection, the closed set $\operatorname{cl}_X\{a\}\subseteq A$ meets every $K$. Minimality gives $A=\operatorname{cl}_X\{a\}$.

Conversely, suppose every closed Rudin set is a point closure. If $\bigcap\mathcal K\subseteq U$ but no $K\in\mathcal K$ is contained in the open set $U$, the first paragraph gives a closed Rudin set $A\subseteq X\setminus U$ meeting every $K$. Write $A=\operatorname{cl}_X\{a\}$. If $b\in K\cap A$, then $b\leq_X a$ in the specialization order, so saturation gives $a\in K$. Thus $a\in\bigcap\mathcal K\subseteq U$, a contradiction.
\end{proof}

\section{A non-sober dcpo with countably many ideals}\label{sec:fan}
\subsection{Construction and directed completeness}
Write $\N=\{0,1,2,\ldots\}$ and
\[
 R_i(k)=\{(i,j):j\geq k\}\subseteq\N^2.
\]
For $p=(n,m)$, a finite set $J\subseteq\{n+1,n+2,\ldots\}$, and a threshold vector $\mathbf k=(k_i)_{i\in J}\in\N^J$, put
\begin{equation}\label{fan:fan}
 T(p;J,\mathbf k)=\{p\}\cup\bigcup_{i\in J}R_i(k_i).
\end{equation}
Let $\RR$ be the collection of all sets of the form \eqref{fan:fan}, ordered by reverse inclusion:
\[
 S\leq T~\Longleftrightarrow~ T\subseteq S.
\]
The set $J$ may be empty, but $\emptyset\notin\RR$. Every $S\in\RR$ has exactly one point in its lowest occupied row. Denote this \emph{anchor} by $a(S)$, its row by $h(S)$, and the finite set of occupied rows by $\rows(S)$. These data are unique, so $\RR$ is countably infinite.

\begin{lemma}\label{fan:anchor}
Every directed family $\mathcal D\subseteq\RR$ has a cofinal directed subfamily whose members have the same anchor and are all contained in one fixed member of $\mathcal D$.
\end{lemma}
\begin{proof}
Choose $S_0\in\mathcal D$. The family
\[
 \mathcal D^0=\{S\in\mathcal D:S\subseteq S_0\}
\]
is directed and cofinal. The possible anchor rows of its members belong to the finite set $\rows(S_0)$. Choose $S_1\in\mathcal D^0$ with maximal anchor row, and put
\[
 \mathcal D'=\{S\in\mathcal D:S\subseteq S_1\}.
\]
This family is again directed and cofinal. For $S\in\mathcal D'$, inclusion gives $h(S)\geq h(S_1)$, while maximality gives equality. The row $h(S_1)$ contains only the anchor of $S_1$, so $a(S)=a(S_1)$. All row supports are contained in $\rows(S_1)$.
\end{proof}

\begin{lemma}\label{fan:dcpo}
The poset $\RR$ is a dcpo, and
\[
 \sup\mathcal D=\bigcap\mathcal D
\]
for every directed family $\mathcal D\subseteq\RR$.
\end{lemma}
\begin{proof}
Use Lemma~\ref{fan:anchor}, and let $p$ be the common anchor of $\mathcal D'$. In each higher row of the fixed finite support, a member has either empty trace or a full tail. If some member has empty trace in that row, so does the intersection. Otherwise the intersection is a full tail when the thresholds are bounded, and is empty when they are unbounded. The anchor $p$ survives. Hence $\bigcap\mathcal D'$ again has the form \eqref{fan:fan}. Cofinality gives $\bigcap\mathcal D=\bigcap\mathcal D'$. This intersection is an upper bound in the reverse-inclusion order, and any other upper bound is contained in every member of $\mathcal D$, hence in their intersection.
\end{proof}

\subsection{Meet-continuity}
We verify \eqref{def:mc} by constructing directed families of common lower bounds.

\begin{proposition}\label{fan:mc}
The dcpo $\RR$ is meet-continuous. More precisely, whenever $S\leq\sup\mathcal D$, there is a directed subset of ${\downarrow} S\cap{\downarrow}\mathcal D$ with supremum $S$.
\end{proposition}
\begin{proof}
Put $S_\infty=\sup\mathcal D=\bigcap\mathcal D$; the hypothesis is $S_\infty\subseteq S$. Choose $\mathcal D'$ as in Lemma~\ref{fan:anchor}. Its common anchor belongs to $S_\infty$, hence to $S$. For every $T\in\mathcal D'$, the union $S\cup T$ again has the form~\eqref{fan:fan}: the anchor of $T$ already lies in $S$, and the union of two tails in one row is the tail with the smaller threshold. Hence
\[
 \mathcal E=\{S\cup T:T\in\mathcal D'\}
\]
is a directed family in $\RR$. Each member is below both $S$ and a member of $\mathcal D$. By Lemma~\ref{fan:dcpo},
\[
 \sup\mathcal E
 =\bigcap_{T\in\mathcal D'}(S\cup T)
 =S\cup\bigcap\mathcal D'
 =S\cup S_\infty=S.
\]
This proves the closure condition. No unrestricted finite-union operation on $\RR$ is used.
\end{proof}

\subsection{Non-sobriety}
\begin{lemma}\label{fan:irr}
Any two nonempty Scott open subsets of $\RR$ contain a common singleton.
\end{lemma}
\begin{proof}
Let $U,V$ be nonempty Scott open sets. Choose $S\in U$ and $T\in V$, with anchors $p=(n,m)$ and $q=(r,s)$. Upperness gives $\{p\}\in U$ and $\{q\}\in V$. Fix $i>\max\{n,r\}$. For each $k\in\N$, put
\[
 A_k=\{p\}\cup R_i(k),
 ~
 B_k=\{q\}\cup R_i(k).
\]
Both $(A_k)_{k\in\N}$ and $(B_k)_{k\in\N}$ are increasing sequences in the reverse-inclusion order on $\RR$. By Lemma~\ref{fan:dcpo},
\[
 \sup_{k\in\N}A_k=\{p\},
 ~
 \sup_{k\in\N}B_k=\{q\}.
\]
Since $U$ and $V$ are Scott open, there exist $k_U,k_V\in\N$ such that $A_{k_U}\in U$ and $B_{k_V}\in V$. If $j\geq\max\{k_U,k_V\}$, then $\{(i,j)\}$ is an upper bound of both $A_{k_U}$ and $B_{k_V}$ and hence belongs to $U\cap V$.
\end{proof}

\begin{theorem}\label{fan:main}
The countably infinite poset $\RR$ is a meet-continuous dcpo whose Scott space is not sober.
\end{theorem}
\begin{proof}
Directed completeness and meet-continuity were proved in Lemma~\ref{fan:dcpo} and Proposition~\ref{fan:mc}. Lemma~\ref{fan:irr} says that the whole space $\RR$ is irreducible. It is a nonempty Scott closed subset of itself, but it is not a principal ideal because $\RR$ has no greatest element: a greatest element would be a nonempty set contained in every singleton of $\N^2$. Point closures in a Scott space are principal ideals, so this irreducible closed set has no generic point.
\end{proof}

\subsection{A normal form for all ideals}
For $S\in\RR$, let $G$ be a finite subset of
\[
 \{i>h(S):i\notin\rows(S)\}.
\]
For $t\in\N$, set
\[
 S_G(t)=S\cup\bigcup_{i\in G}R_i(t),
 ~
 \mathcal J(S,G)={\downarrow}\{S_G(t):t\in\N\}.
\]
These are elements of $\RR$, the chain $S_G(t)$ increases to $S$, and $\mathcal J(S,G)$ is an ideal. Its exact membership description is
\begin{equation}\label{fan:membership}
 \mathcal J(S,G)
 =\{T\in\RR:S\subseteq T,\ G\subseteq\rows(T)\}.
\end{equation}
Indeed, the forward implication follows from $S_G(t)\subseteq T$ for some $t$. Conversely, if $S\subseteq T$, then $h(T)\leq h(S)$. Every row in $G$ is therefore represented in $T$ by a full tail, not by the anchor. A common upper bound for the finitely many tail thresholds gives $S_G(t)\subseteq T$.

\begin{theorem}\label{fan:ideals}
Every ideal $I$ of $\RR$ has a unique representation $I=\mathcal J(S,G)$ as above. Here $S=\sup I$, and
\begin{equation}\label{fan:recoverG}
 G=\left(\bigcap_{T\in I}\rows(T)\right)\setminus\rows(S).
\end{equation}
In particular, $\mathcal J(S,G)$ is principal if and only if $G=\emptyset$, and $\RR$ has countably infinitely many ideals.
\end{theorem}
\begin{proof}
Apply Lemma~\ref{fan:anchor} to the directed set $I$. There is $T_1\in I$ such that
\[
 I' =\{T\in I:T\subseteq T_1\}
\]
is cofinal and has a fixed anchor $p$. All higher rows of members of $I'$ belong to the finite set $\rows(T_1)\setminus\{h(T_1)\}$. Put
\[
 J_\infty=\left(\bigcap_{T\in I'}\rows(T)\right)\setminus\{h(T_1)\}.
\]
Thus $J_\infty$ consists precisely of the rows strictly above the common anchor that occur in every member of $I'$. For each row in $\rows(T_1)\setminus(\{h(T_1)\}\cup J_\infty)$, choose a member in which it is absent. Directedness combines these finitely many choices into $T_2\in I'$. After restricting to the cofinal family of members contained in $T_2$, the occupied rows are exactly $\{h(T_1)\}\cup J_\infty$.

For each $i\in J_\infty$, the corresponding tail thresholds have a supremum in $\N\cup\{\infty\}$. Let $J_{\mathrm{bd}}$ consist of the rows with finite supremum $k_i$, and put $G=J_\infty\setminus J_{\mathrm{bd}}$. The finite integer suprema are attained, and directedness attains all of them simultaneously. Restrict to a further cofinal family $I''$ in which every $J_{\mathrm{bd}}$-threshold equals $k_i$. Each row in $G$ remains present with unbounded threshold. Define
\[
 S=\{p\}\cup\bigcup_{i\in J_{\mathrm{bd}}}R_i(k_i).
\]
The intersection formula gives $S=\sup I$, and $G$ is admissible for $S$.

Fix $t\in\N$. In each row of $G$, choose a member of $I''$ whose threshold is at least $t$, and combine these finitely many choices by directedness. For the resulting $T$, we have $T\subseteq S_G(t)$, so $S_G(t)\leq T\in I$ and lowerness gives $S_G(t)\in I$.

Conversely, for $T\in I$, choose $T'\in I''$ with $T'\subseteq T$. If $t$ is at least all $G$-thresholds of $T'$, then
\[
 S_G(t)\subseteq T'\subseteq T,
\]
so $T\leq S_G(t)$. Hence $I=\mathcal J(S,G)$.

The supremum recovers $S$. Formula~\eqref{fan:membership} shows that every member of $I$ occupies each row of $\rows(S)\cup G$, and each $S_G(t)$ occupies exactly those rows. This proves \eqref{fan:recoverG} and uniqueness. If $G=\emptyset$, then $\mathcal J(S,G)={\downarrow} S$. If $G\neq\emptyset$, its supremum $S$ is not a member, so the ideal is nonprincipal. There are only countably many admissible pairs $(S,G)$, and the distinct principal ideals show that the number of ideals is countably infinite.
\end{proof}

The canonical chain $\{S_G(t):t\in\N\}$ belongs to the ideal; it need not be a subset of an originally supplied directed generating family. If $G\neq\emptyset$, the chains with thresholds $t$ and $2t$ generate the same ideal but are different subsets. Different ideals can also have the same supremum: for $S=\{(0,0)\}$, the ideals $\mathcal J(S,\{1\})$, $\mathcal J(S,\{2\})$ and ${\downarrow} S$ are distinct. Thus the persistent-row parameter $G$ cannot be omitted.

\subsection{Finite powers}
Miao, Xi, Li and Zhao proved that if two posets have only countably many incremental ideals (their term for nonprincipal ideals), then the Scott topology on their order product is the ordinary product topology~\cite[Lemma~4.1, pp.~827--829]{MXLZ}.

For a finite product $Q_1\times\cdots\times Q_r$, every ideal $I$ equals
\[
 I=\prod_{j=1}^r\pi_j(I).
\]
Each projection is an ideal. Conversely, given one point in each projection, directedness of $I$ provides a tuple above finitely many coordinate witnesses; lowerness then puts the prescribed tuple in $I$. Hence a finite product of posets with countably many ideals again has countably many ideals.

\begin{corollary}\label{fan:finitepowers}
For every integer $r\geq1$,
\[
 \sigma(\RR^r)=\prod_{j=1}^r\sigma(\RR).
\]
\end{corollary}
\begin{proof}
Theorem~\ref{fan:ideals} gives countably many ideals in $\RR$, and the preceding product observation gives countably many ideals in every finite power. Induction using the product criterion of Miao, Xi, Li and Zhao gives the topology equality.
\end{proof}

Countability of ideals does not enumerate all irreducible closed sets: the whole space $\RR$ is irreducible but is not directed, since two distinct singletons have no common upper bound.

\subsection{Compact sets, finite-row truncations and well-filteredness}
For $N\in\N$, put
\[
 E_N=\{0,\ldots,N\}\times\N,
 ~
 \RR_{\leq N}=\{S\in\RR:S\subseteq E_N\}.
\]
The compactness statement below does not assume saturation.

\begin{lemma}\label{fan:compact-rows}
Every compact subset of $\Sigma\RR$ is contained in $\RR_{\leq N}$ for some $N\in\N$.
\end{lemma}
\begin{proof}
Suppose that a compact set $K$ is not contained in any $\RR_{\leq N}$. Choose $S_n\in K$ and $p_n\in S_n$ so that the row indices of $p_n$ are strictly increasing. For $m\in\N$, let
\[
 F_m=\{p_n:n\geq m\},
 ~
 C_m=\{S\in\RR:S\cap F_m\ne\emptyset\}.
\]
Each $C_m$ is lower in the reverse-inclusion order. To check Scott closedness, let $\mathcal D\subseteq C_m$ be directed and fix $S_0\in\mathcal D$. The family $\mathcal D_0=\{S\in\mathcal D:S\subseteq S_0\}$ is directed and cofinal. The set $F=S_0\cap F_m$ is finite, because $S_0$ occupies finitely many rows and $F_m$ contains at most one point in any row. The nonempty subsets $S\cap F$, for $S\in\mathcal D_0$, form a filtered family in the finite set $F$, and hence have nonempty intersection. Lemma~\ref{fan:dcpo} gives
\[
 (\sup\mathcal D)\cap F_m
 =\Bigl(\bigcap\mathcal D_0\Bigr)\cap F_m\ne\emptyset.
\]
Thus $C_m$ is Scott closed. The sets $C_m$ decrease and all meet $K$, since $S_m\in K\cap C_m$. However, their intersection is empty: any fixed row-tail set $S\in\RR$ meets only finitely many of the distinct rows containing the $p_n$. This contradicts compactness of $K$.
\end{proof}

\begin{lemma}\label{fan:truncations}
Each $\RR_{\leq N}$ is an algebraic dcpo. Its Scott topology is the subspace topology inherited from $\Sigma\RR$.
\end{lemma}
\begin{proof}
Directed suprema in $\RR_{\leq N}$ are the same intersections as in $\RR$. Consider an element containing a tail in every row strictly above its anchor up to row $N$:
\[
 T=\{(r,s)\}\cup\bigcup_{r<i\leq N}R_i(k_i).
\]
We show that $T$ is compact as an element of $\RR_{\leq N}$. Let $\mathcal D\subseteq\RR_{\leq N}$ be directed with $T\leq\sup\mathcal D$, equivalently $\bigcap\mathcal D\subseteq T$. By Lemma~\ref{fan:anchor}, restrict to a cofinal family with a fixed anchor $q$ and a fixed finite containing set. The anchor $q$ belongs to $T$. In each row above $q$, the intersection of the traces is contained in the corresponding tail of $T$. Hence some member already has its trace contained in that tail: either the row disappears, or a tail threshold reaches the required finite value. There are only finitely many rows. Directedness combines these choices into a member $S\in\mathcal D$ with $S\subseteq T$, that is, $T\leq S$.

For arbitrary $S\in\RR_{\leq N}$, fill all missing rows above its anchor with tails of threshold $n$:
\[
 S^{(n)}=S\cup
 \bigcup_{\substack{h(S)<i\leq N\\i\notin\rows(S)}}R_i(n).
\]
Each $S^{(n)}$ is compact by the preceding paragraph, and $S^{(n)}\leq S^{(n+1)}\leq S$ with $\sup_n S^{(n)}=S$. Consequently every element is a directed supremum of compact elements below it, so $\RR_{\leq N}$ is algebraic.

The inclusion into $\RR$ preserves directed suprema, so every ambient Scott open set has a Scott open trace. Conversely, for $U\in\sigma(\RR_{\leq N})$ define
\[
 \widetilde U
 =\{S\in\RR:h(S)>N\}
 \ \cup\
 \{S\in\RR:h(S)\leq N,\ S\cap E_N\in U\}.
\]
This is an upper set and its trace on $\RR_{\leq N}$ is $U$. If $\sup\mathcal D\in\widetilde U$, use Lemma~\ref{fan:anchor} to fix the anchor on a cofinal subfamily. When its row exceeds $N$, that subfamily is already in $\widetilde U$. Otherwise its intersections with $E_N$ form a directed family in $\RR_{\leq N}$ with supremum $(\sup\mathcal D)\cap E_N\in U$. Scott openness of $U$ yields a member of $\mathcal D$ in $\widetilde U$. Thus $\widetilde U$ is Scott open, proving the subspace assertion.
\end{proof}

\begin{theorem}\label{fan:wellfiltered}
The Scott space $\Sigma\RR$ is well-filtered but is neither coherent nor sober. It satisfies the Scott product equality for every finite power.
\end{theorem}
\begin{proof}
Let $\mathcal K\subseteq\mathcal Q(\Sigma\RR)$ be filtered and let $U\in\sigma(\RR)$ satisfy $\bigcap\mathcal K\subseteq U$. Fix $K_0\in\mathcal K$. By Lemma~\ref{fan:compact-rows}, $K_0\subseteq\RR_{\leq N}$ for some $N$. The subfamily
\[
 \mathcal K_0=\{K\in\mathcal K:K\subseteq K_0\}
\]
is filtered and cofinal under reverse inclusion, so it has the same intersection as $\mathcal K$. By Lemma~\ref{fan:truncations}, each member is compact and saturated in $\Sigma(\RR_{\leq N})$. This algebraic Scott space is sober, hence well-filtered~\cite[Section~II-1]{GHK}. Therefore some $K\in\mathcal K_0$ is contained in $U\cap\RR_{\leq N}$, proving well-filteredness of $\Sigma\RR$.

For noncoherence, take $S=\{(0,0)\}\cup R_1(0)$ and $T=\{(0,1)\}\cup R_1(0)$. Their principal upper sets are compact saturated, but
\[
 {\uparrow} S\cap{\uparrow} T
 =\bigl\{\{(1,j)\}:j\in\N\bigr\}.
\]
In $\RR_{\leq1}$ every $\{(1,j)\}$ is a compact maximal element, so this subspace is infinite discrete and not compact. By the subspace equality in Lemma~\ref{fan:truncations}, it is not compact in $\Sigma\RR$ either. Non-sobriety and the finite-power equalities are Theorem~\ref{fan:main} and Corollary~\ref{fan:finitepowers}.
\end{proof}

\section{Failure of the Scott product equality}\label{sec:word}
Let $\HH$ be the set of all finite words over the alphabet
$\N^{<\omega}\setminus\{\emptyset\}$. Its alphabet elements are called \emph{letters}; thus a letter is itself a nonempty finite string. Write $\ell(w)$ for the number of letters in a word $w$, and $|s|$ for the length of a letter $s$ as a natural-number string. The empty word is denoted by $\varepsilon$, while the empty string $\emptyset$ is used only as the root of a later recursion. We write $\langle s\rangle$ for the one-letter word with letter $s$, and write $s\preceq t$ when the letter $s$ is a prefix of the letter $t$. Define
\begin{equation}\label{word:eq:order}
\begin{aligned}
(u_1,\ldots,u_m)\leq(v_1,\ldots,v_n)
~\Longleftrightarrow~
&\exists\,1\leq i_1<\ldots<i_n\leq m\\
&\text{such that }u_{i_j}\preceq v_j~(1\leq j\leq n).
\end{aligned}
\end{equation}
A lower word may therefore contain additional letters and shorter matched letters. Let $X=2^{\N}$ with its usual product topology. Subsection~\ref{word:subsec:labels} assigns a clopen label $K_s\subseteq X$ to every letter. For a word $w=(w_1,\ldots,w_m)$, put
\[
 K(w)=\bigcup_{i=1}^m K_{w_i},
 ~
 K(\varepsilon)=\emptyset.
\]

\begin{theorem}\label{word:thm:main}
The poset $\HH$ is a countable meet-continuous dcpo with greatest element $\varepsilon$. The set
\begin{equation}\label{word:eq:Omega}
 \Omega=\{(u,v)\in\HH^2:K(u)\cup K(v)\neq X\}
\end{equation}
is Scott open for the coordinatewise order, contains $(\varepsilon,\varepsilon)$, and contains no ordinary product neighbourhood of that point. Hence
\[
 \sigma(\HH)\times\sigma(\HH)
 \subsetneq
 \sigma(\HH\times\HH).
\]
\end{theorem}
This gives a negative answer to Question~1.1 of Xu~\cite{Xu}.

\subsection{The order and directed suprema}
Countability is immediate. Identity embeddings, composition, and the fact that opposite comparisons force equal lengths and identity position maps verify the partial-order axioms. The empty word is greatest.

\begin{lemma}\label{word:lem:sup}
Every directed subset of $\HH$ has a supremum. If an ascending cofinal sequence has eventually constant word length, then its supremum is obtained by retaining, in their original order, exactly the letter coordinates whose internal lengths eventually stabilize.
\end{lemma}
\begin{proof}
Enumerate a directed set $\mathcal D$ and successively choose in $\mathcal D$ an upper bound of the previous choice and the next enumerated element. The resulting ascending sequence is cofinal. Word lengths are nonincreasing, so a tail has fixed length $\ell$. Write that tail as
\[
 w^k=(s_1^k,\ldots,s_\ell^k),
 ~
 s_i^k\preceq s_i^{k+1}.
\]
A letter coordinate of bounded internal length stabilizes. Retain exactly those coordinates, in their original order, and let $w^\infty$ be the word formed by their stable values. Matching the retained positions shows $w^k\leq w^\infty$.

Let $v$ be another upper bound. There are finitely many increasing position maps that can witness $w^k\leq v$. One map occurs on a cofinal subsequence. Every coordinate used by that map is a prefix of a fixed finite letter of $v$ on the subsequence, and monotonicity makes its internal length bounded on the whole tail. Its stable value still prefixes the corresponding letter of $v$. Restricting the same map to the retained coordinates gives $w^\infty\leq v$. The cofinal sequence and $\mathcal D$ have the same upper bounds. The argument includes $\ell=0$, the case of no retained coordinates, and eventually constant sequences.
\end{proof}

\begin{lemma}\label{word:lem:mc}
The dcpo $\HH$ is meet-continuous. More precisely, if $z\leq\sup\mathcal D$, then an ascending sequence in ${\downarrow} z\cap{\downarrow}\mathcal D$ has supremum $z$.
\end{lemma}
\begin{proof}
Use Lemma~\ref{word:lem:sup}. Discard a finite initial segment so that all stable coordinates have reached their final values, and let their positions in $w^k$ be $i_1<\ldots<i_r$. Fix a matching of $w^\infty=\sup\mathcal D$ into $z$ at positions $j_1<\ldots<j_r$; then $z_{j_t}\preceq w^k_{i_t}$.

Amalgamate the two finite linear orders of positions along these anchors. In every interval before, between or after anchors, first list the unmatched positions of $z$, and then the positions of $w^k$ in that interval. The latter positions are all unbounded coordinates, because every stable coordinate is an anchor. At anchor $t$, output the letter $z_{j_t}$. Denote the resulting word by $e_k$.

The copy of the positions of $z$ gives an increasing embedding witnessing $e_k\leq z$. To witness $e_k\leq w^k$, match every unbounded coordinate to its inserted copy and every stable coordinate to the corresponding anchor, whose output letter is a prefix. The pattern is independent of $k$, so $(e_k)$ is ascending. Only the inserted unbounded coordinates vary, and they disappear at the supremum; all positions copied from $z$ remain in their original order. Lemma~\ref{word:lem:sup} gives $\sup_k e_k=z$. If $r=0$, take the concatenation $e_k=zw^k$.
\end{proof}

\begin{lemma}\label{word:lem:concat}
The concatenation map $(u,v)\mapsto uv$ is monotone and preserves directed suprema of the order product. For a fixed word $u$ and a branch $\beta\in\N^\N$, where $\beta|k$ denotes the initial segment of length $k$,
\begin{equation}\label{word:eq:insertion}
 u\langle\beta|1\rangle\leq u\langle\beta|2\rangle\leq\ldots,
 ~
 \sup_k u\langle\beta|k\rangle=u.
\end{equation}
\end{lemma}
\begin{proof}
Matching maps concatenate. For directed suprema, choose a cofinal sequence on which both word lengths are constant. Coordinatewise retention and deletion commute with concatenation. Formula~\eqref{word:eq:insertion} is the one-transient-coordinate case of Lemma~\ref{word:lem:sup}.
\end{proof}

\subsection{Clopen labels with universal nowhere dense residuals}\label{word:subsec:labels}
The labels are designed so that an indefinitely extended letter leaves only a nowhere dense residual set, while every closed nowhere dense set can be retained along a suitable branch. The first property will make the covering language Scott closed; the second will be used to rule out ordinary open rectangles.

For a finite binary string $b$, let $[b]\subseteq X$ be its cylinder. Enumerate all basic cylinders as $(B_n)_{n\geq0}$, with $B_0=X$, and enumerate all clopen subsets as $(V_j)_{j\geq0}$, allowing repetitions. Choose the enumeration so that $V_0=[0]$ and $V_1=[1]$.

For every finite natural-number string, including the recursion root, write $s^\frown j$ for the string obtained by appending $j$ to $s$, and define
\begin{equation}\label{word:eq:labels}
 K_\emptyset=X,
 ~
 K_{s^\frown j}=\begin{cases}
 K_s\cap V_j,&B_{|s|}\not\subseteq K_s\cap V_j,\\
 \emptyset,&B_{|s|}\subseteq K_s\cap V_j.
 \end{cases}
\end{equation}
The labels are clopen, and $s\preceq t$ implies $K_t\subseteq K_s$. Every label indexed by a nonempty string is proper because the first step excludes $B_0=X$. Empty labels are allowed; the corresponding letters remain in $\HH$.

\begin{lemma}\label{word:lem:labels}
The labels have the following properties.
\begin{enumerate}[\rm (1)]
\item For every $\alpha\in\N^\N$, the set
      $\mathcal N_\alpha=\bigcap_{k\geq1}K_{\alpha|k}$ is closed nowhere dense, possibly empty.
\item For every closed nowhere dense set $\mathcal N\subseteq X$, some $\alpha\in\N^\N$ satisfies
      $\mathcal N\subseteq K_{\alpha|k}$ for all $k\geq1$.
\item For every nonempty basic cylinder $B$, some $\beta\in\N^\N$ satisfies
      $\emptyset\neq K_{\beta|k}\subseteq B$ for every $k\geq1$.
\end{enumerate}
\end{lemma}
\begin{proof}
For (1), $B_n$ is not contained in $K_{\alpha|(n+1)}$. Since that label is closed, its complement meets $B_n$ in a nonempty open set. Hence a subcylinder of $B_n$ is disjoint from $\mathcal N_\alpha$. Thus $\mathcal N_\alpha$ contains no nonempty basic cylinder and is closed.

For (2), suppose $\mathcal N\subseteq K_s$ and $|s|=n$. Choose a nonempty basic clopen cylinder $B^\prime\subseteq B_n\setminus\mathcal N$ and put $W=K_s\setminus B^\prime$. Then $\mathcal N\subseteq W\subseteq K_s$ and $B_n\not\subseteq W$. Choose $j$ with $V_j=W$. Formula~\eqref{word:eq:labels} gives $K_{s^\frown j}=W$. Iteration constructs the required branch, including the case $\mathcal N=\emptyset$.

For (3), choose $x\in B$ and a proper clopen set $W$ with $x\in W\subseteq B$. Use $W$ as the first label, and then continue the construction in (2) while retaining the singleton $\{x\}$. A singleton is closed nowhere dense because Cantor space has no isolated points.
\end{proof}
Fix a branch $\beta_n$ as in part (3) for each $B_n$. These choices are made before any Scott neighbourhoods are considered.

\subsection{A Scott closed covering language}
Define
\begin{equation}\label{word:eq:cover}
 \mathcal C=\{w\in\HH:K(w)=X\},
 ~
 \mathcal O=\HH\setminus\mathcal C.
\end{equation}
The word $\langle(0),(1)\rangle$ belongs to $\mathcal C$, whereas $\varepsilon\notin\mathcal C$. If $u\leq v$, each label in $v$ is contained in the label of its matched source letter, so $K(u)\supseteq K(v)$. Hence $\mathcal C$ is lower.

\begin{lemma}\label{word:lem:closed}
The set $\mathcal C$ is Scott closed. Consequently, $\mathcal O$ is a Scott open neighbourhood of $\varepsilon$.
\end{lemma}
\begin{proof}
Take an eventually fixed-length cofinal sequence
$w^k=(s_1^k,\ldots,s_\ell^k)$ in a directed subset of $\mathcal C$. The labels decrease coordinatewise, and every word $w^k$ satisfies $K(w^k)=X$. Thus
\begin{equation}\label{word:eq:finiteinter}
 X=\bigcup_{i=1}^{\ell}\bigcap_kK_{s_i^k}.
\end{equation}
Indeed, for each $x\in X$, the sets $\{i:x\in K_{s_i^k}\}$ form a decreasing sequence of nonempty subsets of the finite set $\{1,\ldots,\ell\}$, and therefore have nonempty intersection.

Stable coordinates leave clopen labels. Each unbounded coordinate is cofinal among the prefixes of a branch $\alpha$, so its label intersection is the nowhere dense set $\mathcal N_\alpha$ from Lemma~\ref{word:lem:labels}(1). Let $K_{\mathrm{st}}$ be the finite union of the stable labels. If $X\setminus K_{\mathrm{st}}$ were nonempty, it would be a nonempty open set covered by finitely many closed nowhere dense sets. Successively choosing a smaller nonempty open set outside each of them gives a contradiction. Hence $K_{\mathrm{st}}=X$. The supremum word retains precisely the stable letters and therefore belongs to $\mathcal C$. This also excludes an all-unbounded directed sequence in $\mathcal C$.
\end{proof}

\subsection{Failure of ordinary product openness}
By Lemma~\ref{word:lem:concat},
\[
 \Omega=\{(u,v)\in\HH^2:uv\in\mathcal O\}
\]
is Scott open in the order square, and $(\varepsilon,\varepsilon)\in\Omega$. Let $V,W\in\sigma(\HH)$ contain $\varepsilon$. Begin with $u_{-1}=\varepsilon$. Once $u_{n-1}\in V$ is chosen, the chain in \eqref{word:eq:insertion} with $\beta=\beta_n$ has supremum $u_{n-1}$. Scott openness gives $k_n\geq1$ such that
\begin{equation}\label{word:eq:left}
 s_n=\beta_n|k_n,
 ~
 u_n=u_{n-1}\langle s_n\rangle\in V.
\end{equation}
Thus every word $u_n=\langle s_0,\ldots,s_n\rangle$ belongs to the same open set $V$, and
$\emptyset\neq K_{s_n}\subseteq B_n$. Put
\[
 \mathcal N=X\setminus\bigcup_{n\geq0}K_{s_n}.
\]
Every basic cylinder contains one of the nonempty labels, so $X\setminus\mathcal N$ is dense open and $\mathcal N$ is closed nowhere dense. Lemma~\ref{word:lem:labels}(2) gives $\alpha\in\N^\N$ with
$\mathcal N\subseteq K_{\alpha|k}$ for all $k$.

The one-letter chain $\langle\alpha|k\rangle$ has supremum $\varepsilon$. Scott openness of $W$ supplies $k\geq1$ with $v=\langle t\rangle\in W$, where $t=\alpha|k$. Since $K_t$ is clopen, $X\setminus K_t$ is a compact subset of $X\setminus\mathcal N$. Choose a finite set $M\subseteq\N$ such that
\[
 X\setminus K_t\subseteq\bigcup_{n\in M}K_{s_n}.
\]
Take $m\geq\max M$, with $m=0$ if $M=\emptyset$. Then
\begin{equation}\label{word:eq:rectangle}
 K(u_m)\cup K_t=X,
 ~
 (u_m,v)\in(V\times W)\setminus\Omega.
\end{equation}
The argument also covers $\mathcal N=\emptyset$. Thus no ordinary product neighbourhood of $(\varepsilon,\varepsilon)$ is contained in $\Omega$. The reverse topology inclusion follows from Scott continuity of the coordinate projections, completing Theorem~\ref{word:thm:main}.

\begin{remark}[Two sequence directions]
For fixed $n$, the insertion sequence used to choose $s_n$ is ascending and has supremum $u_{n-1}$. Across $n$, the selected words $u_n$ are descending. No directed supremum of that descending sequence is used. Compactness is applied in the Cantor space $X$, where the finite unions $K(u_n)$ are increasing.
\end{remark}

For the open set $\mathcal O$ in \eqref{word:eq:cover},
\[
 (u,v)\in\Omega
 ~\Longleftrightarrow~
 \text{some }z\in\mathcal O\text{ satisfies }z\leq u,v.
\]
Indeed, if $(u,v)\in\Omega$, the concatenation $uv$ is such a common lower bound. Conversely, $z\leq u,v$ implies $K(z)\supseteq K(u)\cup K(v)$, so $K(u)\cup K(v)=X$ would force $z\notin\mathcal O$. Thus $\Omega=\mathrm{LB}_{\mathcal O}$.

The set $\Omega$ has empty interior in the ordinary product topology, and its complement is dense there, although the complement is Scott closed in the order square. Every nonempty Scott open subset of $\HH$ contains the greatest element $\varepsilon$; hence every nonempty set open in the ordinary product topology contains $(\varepsilon,\varepsilon)$ and cannot be contained in $\Omega$.

\subsection{Finite generation of common upper bounds}
The well-filteredness of the word example follows from a different feature of its order.

\begin{lemma}\label{word:finite-upper-bounds}
For all $u,v\in\HH$, there is a finite set $F(u,v)\subseteq\HH$ such that
\[
 {\uparrow} u\cap{\uparrow} v={\uparrow} F(u,v).
\]
\end{lemma}
\begin{proof}
Write $u=(u_1,\ldots,u_m)$ and $v=(v_1,\ldots,v_n)$. For $0\leq k\leq\min(m,n)$ choose increasing index lists $i_1<\cdots<i_k$ and $j_1<\cdots<j_k$ such that $u_{i_\ell}$ and $v_{j_\ell}$ are prefix comparable for every $\ell$. Let $c_\ell$ be the longer of these two letters and put $c=(c_1,\ldots,c_k)$. The case $k=0$ gives $c=\varepsilon$. There are only finitely many such choices; let $F(u,v)$ consist of their resulting words. Each resulting $c$ is a common upper bound of $u,v$ by~\eqref{word:eq:order}.

Conversely, if $w=(w_1,\ldots,w_k)$ is a common upper bound, choose the two increasing matching maps witnessing $u\leq w$ and $v\leq w$. The matched letters at each position are both prefixes of $w_\ell$, and therefore are prefix comparable. Their longer letter is still a prefix of $w_\ell$. The word $c$ obtained from these matches belongs to $F(u,v)$ and satisfies $c\leq w$. This proves the equality.
\end{proof}

\begin{proposition}\label{word:wellfiltered}
The Scott space $\Sigma\HH$ is a strong $d$-space, well-filtered and coherent.
\end{proposition}
\begin{proof}
We first recall a useful consequence of the order-theoretic Rudin lemma. Let $P$ be a dcpo such that every ${\uparrow} x\cap{\uparrow} y$ is a finitely generated upper set, allowing the empty set. Then, for every directed $D\subseteq P$, every $x\in P$ and every $U\in\sigma(P)$,
\[
 {\uparrow}\sup D\cap{\uparrow} x\subseteq U
 ~\Longrightarrow~
 {\uparrow} d\cap{\uparrow} x\subseteq U
 \text{ for some }d\in D.
\]
This is the strong $d$-space condition for a Scott space~\cite{XZ,LXLower}. To prove it, suppose that every $K_d={\uparrow} d\cap{\uparrow} x$ meets the lower set $C=P\setminus U$. These $K_d$ form a filtered family of nonempty finitely generated upper sets. Rudin's lemma, in the form of~\cite[Lemma~2.3]{LXCompact}, gives a directed set $E\subseteq C$ such that ${\downarrow} E$ meets every $K_d$. Since $K_d$ is upper, $E$ itself meets every $K_d$. Thus $e=\sup E$ belongs to $C$, bounds $D$ and also bounds $x$. This contradicts ${\uparrow}\sup D\cap{\uparrow} x\subseteq U$.

Apply this observation to $\HH$ using Lemma~\ref{word:finite-upper-bounds}. By Lawson and Xu~\cite[Corollary~64]{LXLower}, the Scott space of a dcpo is well-filtered whenever it is a strong $d$-space. Hence $\Sigma\HH$ is well-filtered. Each ${\uparrow} u\cap{\uparrow} v$ is compact, being a finite union of principal upper sets. The coherence criterion for well-filtered dcpos of Jia, Jung and Li~\cite[Lemma~3.1]{JJLCoherence} now gives coherence.
\end{proof}

Following Xu, Yang and Chen~\cite{XXC}, a $T_0$ space $X$ is a \emph{strong $R$-space} if, whenever a nonempty family $\{K_i:i\in I\}\subseteq\mathcal Q(X)$ and an open set $U$ satisfy $\bigcap_{i\in I}K_i\subseteq U$, there is a nonempty finite $F\subseteq I$ with $\bigcap_{i\in F}K_i\subseteq U$. A space is \emph{strongly well-filtered} if, for a filtered family $\mathcal K\subseteq\mathcal Q(X)$, a fixed $K\in\mathcal Q(X)$ and an open $U$, the inclusion $K\cap\bigcap\mathcal K\subseteq U$ implies $K\cap K_0\subseteq U$ for some $K_0\in\mathcal K$.

\begin{corollary}\label{word:strong-R}
The Scott space $\Sigma\HH$ is a strong $R$-space and hence strongly well-filtered.
\end{corollary}
\begin{proof}
Every coherent well-filtered space is a strong $R$-space~\cite{XXC}. Indeed, the nonempty finite intersections of a family of compact saturated sets form a filtered family of compact saturated sets, unless some finite intersection is empty, in which case the desired finite conclusion is already available. Apply well-filteredness to this family. Every strong $R$-space is strongly well-filtered: apply its defining property to a filtered family together with the fixed compact saturated set, and then choose a member of the filtered family contained in the finitely many selected members. Proposition~\ref{word:wellfiltered} now gives both assertions for $\Sigma\HH$.
\end{proof}

Together with Theorem~\ref{word:thm:main}, this shows that coherence and well-filteredness do not imply the square Scott product equality, even for countable meet-continuous dcpos with a greatest element. The same counterexample is a strong $R$-space by Corollary~\ref{word:strong-R}. In particular, $\HH$ has no countable family of directed subsets complete for Scott openness: such a family would imply the square product equality by~\cite[Proposition~2.4(2)]{Xu}.

\section{Local lattice structure and Scott sobriety}\label{sec:local}
Local lattice operations allow us to deduce sobriety from the square Scott product equality. We first prove this implication without a countability assumption, then apply it to countable local lattice classes and to Jia's core-compactness questions. The final subsection treats well-filteredness through closed Rudin sets.

\subsection{Local joins and the conditional sobriety theorem}\label{local:subsec:joins}
A dcpo $P$ is called an \emph{$L$-dcpo} if every principal ideal
\[
 P_r={\downarrow} r
\]
is a complete lattice in its induced order; see Jia~\cite[Section~4.4]{Jia} and Xu~\cite[Definition~3.1]{Xu}. Write $\wedge_r$ and $\vee_r$ for the operations in $P_r$. A local join may depend on the chosen upper bound $r$; compatible elements are not assumed to have a global join.

Every principal ideal is a sub-dcpo with the subspace Scott topology, and a dcpo is meet-continuous if and only if all of its principal ideals are meet-continuous~\cite[Lemma~3.2]{Xu}. We also use the following elementary extension fact. If $P$ is meet-continuous and $U$ is Scott open in $P_r$, then ${\uparrow}_P U$ is Scott open in $P$. To see this, choose an ambient Scott open set $O$ with $O\cap P_r=U$. If $D$ is directed and $u\in U$ satisfies $u\leq\sup D$, meet-continuity gives
\[
 z\in O\cap{\downarrow} u\cap{\downarrow}D.
\]
Then $z\in U$, and some member of $D$ lies in ${\uparrow}_P U$.

In a meet-continuous complete lattice $L$, binary meet preserves directed suprema on the order square. For a directed set $D\subseteq L^2$,
\[
 \left(\sup_{e\in D}e_1\right)\wedge
 \left(\sup_{e\in D}e_2\right)
 =\sup_{e,d\in D}(e_1\wedge d_2)
 =\sup_{e\in D}(e_1\wedge e_2).
\]
The first equality uses separate meet-continuity twice; the second uses a common upper index in $D$. Binary join preserves directed suprema on the order square of every complete lattice.

\begin{lemma}\label{local:separate}
Let $P$ be a meet-continuous $L$-dcpo, and let $a\neq b$ be minimal upper bounds of $\{x,y\}$. If the Scott topology of $P_a\times P_b$ is its ordinary product topology, then $a$ and $b$ have disjoint Scott open neighbourhoods in $P$.
\end{lemma}
\begin{proof}
For $z\in\{x,y\}$ and $(u,v)\in P_a\times P_b$, define
\[
 \theta_z(u,v)=(z\wedge_a u)\wedge_b v,
 ~
 \Phi(u,v)=\theta_x(u,v)\vee_a\theta_y(u,v).
\]
These expressions are well typed: $z\wedge_a u\leq z\leq b$, and $\theta_z(u,v)\leq z\leq a$. The preceding directed-supremum calculation shows that every $\theta_z$, and hence $\Phi$, is Scott continuous on the order product. More explicitly, the first meet values along a directed input family lie below $z$ and have the same supremum in $P_a$, $P_b$ and $P$; the meet calculation in $P_b$ and the join calculation in $P_a$ then give preservation of the directed supremum.

Always $\Phi(u,v)\leq u$. If $u$ and $v$ have a common upper bound $r$, let
\[
 c=\theta_x(u,v)\vee_r\theta_y(u,v).
\]
Both terms lie below $u$ and $v$, so $c\leq u,v$. Since $c\leq u\leq a$, leastness of the join in $P_a$ gives $\Phi(u,v)\leq c$. Conversely, $\Phi(u,v)\leq u\leq r$, so leastness in $P_r$ gives $c\leq\Phi(u,v)$. Thus
\[
 \Phi(u,v)=c\leq v\leq b.
\]
This comparison uses the common intermediate upper bound $u$; it does not assert that arbitrary local joins are independent of their host.

Minimality gives $x\vee_a y=a$ and $a\nleq b$, hence $\Phi(a,b)=a\nleq b$. Therefore the set
\[
 \{(u,v)\in P_a\times P_b:\Phi(u,v)\nleq b\}
\]
is Scott open and contains $(a,b)$. The assumed local product equality gives an open rectangle $U_a\times U_b$ about $(a,b)$ contained in this set. No pair in that rectangle has a common upper bound. The extension fact above now shows that the ambient Scott open sets ${\uparrow}_P U_a$ and ${\uparrow}_P U_b$ are disjoint and contain $a$ and $b$, respectively.
\end{proof}

\begin{theorem}\label{local:sober}
Let $P$ be a meet-continuous $L$-dcpo. If
\[
 \sigma(P\times P)=\sigma(P)\times\sigma(P),
\]
then $\Sigma P$ is sober. 
\end{theorem}
\begin{proof}
For any $a,b\in P$, the set $P_a\times P_b={\downarrow}(a,b)$ is a principal ideal of the order square. The subspace property and the assumed square equality therefore give
\[
\begin{aligned}
 \sigma(P_a\times P_b)
 &=\left.\sigma(P\times P)\right|_{P_a\times P_b}\\
 &=\left.(\sigma(P)\times\sigma(P))\right|_{P_a\times P_b}\\
 &=\sigma(P_a)\times\sigma(P_b).
\end{aligned}
\]
Let $A$ be a nonempty irreducible Scott closed subset of $P$. Suppose that $x,y\in A$ have an upper bound $r\in A$. Their local join $x\vee_r y$ lies in $A$ and is a minimal upper bound of $\{x,y\}$ in $P$: an upper bound below it belongs to $P_r$ and therefore equals it. If $s\in A$ is another upper bound, then $x\vee_s y=x\vee_r y$. Otherwise Lemma~\ref{local:separate} would give disjoint Scott open sets meeting $A$.

The common value is below every upper bound of $x,y$ that lies in $A$; write it as $x\vee_A y$. This notation is used only when an $A$-upper bound exists. No directedness of $A$ is assumed at this stage.

The relation $\mathrm{UB}_A$ introduced in Section~\ref{sec:relations} is contained in $A^2$, because $A$ is lower. It is lower in $A^2$. Let $\mathcal D\subseteq\mathrm{UB}_A$ be directed. The set of all $x\vee_A y$ with $(x,y)\in\mathcal D$ is directed: a coordinatewise upper member of two pairs has an $A$-join above the joins of both pairs. Since $A$ is Scott closed, the supremum $e$ of this set belongs to $A$, and $e$ bounds both coordinates of $\sup\mathcal D$. Hence $\sup\mathcal D\in\mathrm{UB}_A$. Thus $\mathrm{UB}_A$ is Scott closed in the order product $A^2$.

The sub-dcpo $A^2$ is Scott closed in $P^2$, and its intrinsic Scott topology is the subspace topology. The Scott product equality for the square therefore makes $\mathrm{UB}_A$ closed in the ordinary product $A\times A$. It contains the diagonal. Irreducibility makes that diagonal dense in $A\times A$, so $A^2=\mathrm{UB}_A$. Consequently, $A$ is directed. Its supremum belongs to $A$, and lowerness gives $A={\downarrow}\sup A$. The generic point is unique by antisymmetry.
\end{proof}

\begin{corollary}\label{local:countable}
Suppose that $P$ is a countable meet-continuous dcpo and every principal ideal is a join-semilattice in its induced order. Then $\Sigma P$ is sober. The local joins need not agree globally, and $P$ need not have a least element.
\end{corollary}
\begin{proof}
Every principal ideal is a dcpo with binary joins. For a nonempty subset of a principal ideal, the joins of its nonempty finite subsets form a directed family, whose supremum is the supremum of the original subset in that principal ideal.

Adjoin a new least element and write $P_\bot=P\cup\{\bot\}$. Below an old point, the nonempty suprema just constructed, together with the new bottom, give all suprema. Hence $P_\bot$ is an $L$-dcpo. It is countable and meet-continuous by~\cite[Lemma~4.1]{Xu}, and the countable local product theorem~\cite[Theorem~3.4]{Xu} gives
\[
 \sigma(P_\bot\times P_\bot)=\sigma(P_\bot)\times\sigma(P_\bot).
\]
Theorem~\ref{local:sober} therefore makes $\Sigma P_\bot$ sober.

Finally, $\sigma(P_\bot)=\sigma(P)\cup\{P_\bot\}$. If $A\subseteq P$ is nonempty, irreducible and Scott closed, then $A\cup\{\bot\}$ is irreducible and Scott closed in $P_\bot$. Its generic point cannot be the new bottom, since $A$ is nonempty, so it lies in $A$. Restricting its point closure to $P$ shows that $A$ is a principal ideal. Thus $\Sigma P$ is sober.
\end{proof}

Corollary~\ref{local:countable} extends the bounded-complete case of~\cite[Proposition~4.3]{Xu} to the weaker hypothesis of local binary joins.

The local hypothesis does not imply global joins. For example, let the five-element poset have elements $0,a,b,u,v$, cover relations
\[
 0<a,~ 0<b,~ a<u,v,~ b<u,v,
\]
and no other comparabilities. Every principal ideal is a complete lattice, while $a$ and $b$ have local joins $u$ and $v$ below the two maximal bounds and have no global join.

\subsection{Core-compactness and Jia's questions}\label{local:subsec:jia}
Recall that a topological space $X$ is \emph{core-compact} if its lattice $\mathcal O(X)$ of open sets is continuous. Thus $\Sigma P$ is core-compact precisely when $\sigma(P)$ is a continuous lattice~\cite[Definition~2.5.1]{Jia}. On page~139 of his thesis, Jia asks whether a core-compact meet-continuous dcpo must be sober~\cite[Question~4.5.8]{Jia}. The next corollary answers this question for $L$-dcpos.

\begin{corollary}[The $L$-dcpo case of Jia's question]\label{local:corecompact}
Let $P$ be a meet-continuous $L$-dcpo. If $\Sigma P$ is core-compact, then $\Sigma P$ is sober and locally compact. No countability assumption is required.
\end{corollary}
\begin{proof}
By Gierz et al.~\cite[Theorem~II-4.13]{GHK}, core-compactness of $\Sigma P$ implies
\[
 \sigma(Q\times P)=\sigma(Q)\times\sigma(P)
\]
for every dcpo $Q$; see also Jia~\cite[Theorem~2.5.7]{Jia}. Taking $Q=P$ gives the hypothesis of Theorem~\ref{local:sober}, which yields sobriety. A core-compact sober space is locally compact by~\cite[Theorem~V-5.6]{GHK}.
\end{proof}

In particular, for a meet-continuous $L$-dcpo, core-compactness and local compactness of the Scott space are equivalent, and either condition implies sobriety. The converse implication from local compactness to core-compactness holds for all spaces~\cite[Examples~I-1.7(5)]{GHK}. The $L$-dcpo assumption is retained throughout this conclusion.

The second part of Jia's Question~4.5.8 concerns the function space $[P\to P]$, the dcpo of Scott-continuous self-maps of $P$ ordered pointwise. It also has an affirmative answer under the same local-lattice hypothesis, and in that case core-compactness of the function space alone suffices.

\begin{corollary}[Function-space version]\label{local:functionspace}
Let $P$ be an $L$-dcpo. If $\Sigma[P\to P]$ is core-compact, then $\Sigma P$ is locally compact and sober. Neither countability nor an additional meet-continuity assumption on $[P\to P]$ is required.
\end{corollary}
\begin{proof}
Fix $p_0\in P$. The constant-map embedding $c:P\to[P\to P]$, defined by $c(x)(y)=x$, and evaluation $e_{p_0}:[P\to P]\to P$, defined by $e_{p_0}(f)=f(p_0)$, preserve directed suprema, and $e_{p_0}\circ c=\operatorname{id}_P$. Thus $\Sigma P$ is a continuous retract of $\Sigma[P\to P]$. Core-compactness passes to retracts~\cite[Proposition~2.7.9]{Jia}, so $\Sigma P$ is core-compact. Moreover, Jia's Theorem~3.4.2 shows that core-compactness of $\Sigma[P\to P]$ implies meet-continuity of $P$~\cite{Jia}. Corollary~\ref{local:corecompact} now applies.
\end{proof}

The same constant-map retraction and preservation of meet-continuity by retracts~\cite[Proposition~3.1.7]{Jia} also give
\[
 [P\to P]\text{ meet-continuous}\quad\Longrightarrow\quad P\text{ meet-continuous}.
\]

\subsection{Closed Rudin sets and well-filteredness}\label{local:subsec:rudin}
Xi and Lawson proved that the Scott space of every complete lattice is well-filtered~\cite{XiLawson}. Consequently every principal ideal of an $L$-dcpo, with its subspace Scott topology, is well-filtered. Miao, Li and Zhao used this fact to prove $k$-bounded well-filteredness of $L$-dcpos~\cite[Lemma~4.6 and Theorem~4.8]{MLZ}. Here $k$-bounded well-filteredness concerns closed Rudin sets whose supremum exists. The following formulation of the same localization argument distinguishes an upper bound from a preassigned supremum and isolates what remains to be proved for full well-filteredness.

\begin{proposition}\label{local:rudin-boundedness}
Let $P$ be an $L$-dcpo. Then every closed Rudin set in $\Sigma P$ that has an upper bound in $P$ is a principal ideal. Consequently, the following are equivalent:
\begin{enumerate}[\rm (1)]
\item $\Sigma P$ is well-filtered;
\item Every closed Rudin set in $\Sigma P$ has an upper bound in $P$.
\end{enumerate}
\end{proposition}
\begin{proof}
Let $A\subseteq{\downarrow} r$ be a closed Rudin set, witnessed by a filtered family $\mathcal K\subseteq\mathcal Q(\Sigma P)$. For every $K\in\mathcal K$, the set $K\cap{\downarrow} r$ is nonempty, is compact as a closed subset of $K$, and is saturated in the subspace $\Sigma({\downarrow} r)$. These intersections form a filtered family, and all meet the closed set $A$. Since ${\downarrow} r$ is a complete lattice, its Scott space is well-filtered. Hence
\[
 A\cap\bigcap_{K\in\mathcal K}(K\cap{\downarrow} r)\ne\emptyset;
\]
otherwise well-filteredness applied to $({\downarrow} r)\setminus A$ would give a member disjoint from $A$. Choose $a$ in this intersection. The ambient Scott closed set ${\downarrow} a\subseteq A$ meets every $K$, so the defining minimality of $A$ gives $A={\downarrow} a$. Notice that the argument never assumed that $A$ already had a supremum in $P$.

If $\Sigma P$ is well-filtered, every closed Rudin set is principal by Lemma~\ref{wf:rudin-characterization}, and therefore bounded. Conversely, if all closed Rudin sets are bounded, the first part makes them principal; the same characterization yields well-filteredness.
\end{proof}

\section{Consequences and open questions}\label{sec:boundaries}
The finite-word example answers Question~1.1 of Xu negatively, and the row-tail example answers Question~1.2 negatively~\cite{Xu}. The row-tail example is well-filtered and also satisfies every finite product equality considered here, so these equalities do not imply sobriety even within the well-filtered class.

The additional compactness properties distinguish the two examples as follows:
\begin{center}
\begin{tabular}{@{}lcc@{}}
\toprule
Property & $\Sigma\RR$ & $\Sigma\HH$\\
\midrule
Well-filtered & Yes & Yes\\
Coherent & No & Yes\\
Scott product equality for the square & Yes & No\\
Scott sober & No & Not determined here\\
\bottomrule
\end{tabular}
\end{center}
The compactness properties in the table follow from Theorem~\ref{fan:wellfiltered} and Proposition~\ref{word:wellfiltered}. The sobriety of $\Sigma\HH$ remains open. More generally, it remains open whether every sober countable meet-continuous dcpo satisfies
\[
 \sigma(P\times P)=\sigma(P)\times\sigma(P).
\]
For countable meet-continuous dcpos with local binary joins, the product equality~\cite[Corollary~4.2]{Xu} and Corollary~\ref{local:countable} give both properties. The five-element example following that corollary shows that this class properly contains the class of countable meet-continuous complete lattices.

The two failures are visible in the relations of Section~\ref{sec:relations}. In $\RR$, put
\[
 S=\{(0,0)\}\cup R_1(0),~
 T_n=\{(0,1)\}\cup R_1(n)~(n\in\N),~
 T=\{(0,1)\}.
\]
Then $(T_n)_{n\in\N}$ is increasing in the reverse-inclusion order, and
\[
 \sup_{n\in\N}T_n=\bigcap_{n\in\N}T_n=T.
\]
For each $n$, the elements $S$ and $T_n$ have a common singleton upper bound, whereas $S$ and $T$ are disjoint and have no common upper bound. Taking the irreducible closed set $A=\RR$, the lower section ${\downarrow}({\uparrow} S\cap A)$ contains every $T_n$ and omits their supremum $T$. In the word example, the bad Scott open set is exactly $\Omega=\mathrm{LB}_{\mathcal O}$. Thus one example fails ordinary rectangle openness, while the other fails directed-limit stability of common upper bounds.

Core-compactness gives a complementary positive boundary. He, Dong and Wang proved that every countable dcpo with core-compact Scott space is quasicontinuous and hence Scott sober~\cite[Theorem~3.3]{HDW}. Corollary~\ref{local:corecompact} instead applies without a cardinality restriction, under meet-continuity and the $L$-dcpo hypothesis. Neither $\Sigma\RR$ nor $\Sigma\HH$ is core-compact: for $\RR$, this follows from non-sobriety and~\cite[Theorem~3.3]{HDW}; for $\HH$, it follows from the failure of the square product equality and~\cite[Theorem~II-4.13]{GHK}. Thus the two counterexamples leave Jia's general core-compactness question open, while Corollaries~\ref{local:corecompact} and~\ref{local:functionspace} resolve its two parts for $L$-dcpos.





\end{document}